\documentclass[a4paper]{amsart}
\usepackage[utf8]{inputenc}
\usepackage{amssymb,amsmath,enumerate,amsthm,url}
\usepackage[all,cmtip]{xy}
\usepackage{graphicx}
\usepackage[numbers]{natbib}
\usepackage{bibentry}
\usepackage[bookmarks=true]{hyperref}

\usepackage{stackengine}

\usepackage{comment}
\usepackage[colorinlistoftodos]{todonotes}

\def\namedlabel#1#2{\begingroup
    #2%
    \def\@currentlabel{#2}%
    \phantomsection\label{#1}\endgroup
}

\usepackage[nameinlink,capitalize]{cleveref}
\crefformat{equation}{#2(#1)#3}

\newcommand\leqdot{\mathrel{\ensurestackMath{%
  \stackengine{-.5ex}{\lessdot}{-}{U}{c}{F}{F}{S}}}}

\providecommand*{\symdif}{\mathrel{\triangle}}

\theoremstyle{plain}
\newtheorem{theorem}{Theorem}[section]
\newtheorem{mainthm}{Main Theorem}[section]
\newtheorem{lemma}[theorem]{Lemma}
\newtheorem{proposition}[theorem]{Proposition}

\newtheorem{fact}[theorem]{Fact}
\newtheorem*{fact*}{Fact}
\crefname{fact}{Fact}{Facts}
\newtheorem{corollary}[theorem]{Corollary}

\newtheorem*{claim*}{Claim}
\theoremstyle{definition}
\newtheorem{definition}[theorem]{Definition}
\newtheorem*{definition*}{Definition}

\newtheorem*{notation*}{Notation}
\theoremstyle{remark}
\newtheorem{remark}[theorem]{Remark}
\newtheorem*{remark*}{Remark}
\newtheorem{example}[theorem]{Example}
\newtheorem*{example*}{Example}

\newtheorem*{note*}{Note}
\newtheorem{question}[theorem]{Question}
\newtheorem*{question*}{Question}
\newtheorem{observation}[theorem]{Observation}

\numberwithin{equation}{section}

\newcommand{\Z}{\mathbb{Z}}
\newcommand{\Q}{\mathbb{Q}}

\newcommand{\VC}{\operatorname{VC}}
\newcommand{\tp}{\operatorname{tp}}

\newcommand{\U}{\mathcal{U}}
\newcommand{\M}{\mathcal{M}}
\newcommand{\opp}{\operatorname{opp}}

\newcommand{\Sk}{\operatorname{Sk}}

\providecommand{\Th}{\operatorname{Th}}
\providecommand{\set}[2]{\{#1 \mid #2\}}
\providecommand{\sequence}[2]{(#1)_{#2}}

\renewcommand{\SS}{\mathcal{P}}

\providecommand{\phiopp}{{\phi^{\opp}}}

\newcommand{\defn}[1]{{\bf #1}}
\newcommand{\recall}[1]{{\it #1}}

\newcommand{\calF}{\mathcal F}

\renewcommand{\L}{\mathcal L}

\newcommand{\Sh}{\text{Sh}}
\newcommand{\ext}{\text{ext}}

\newcommand{\barA}{{\bar A}}

\begin{document}

\title{On large externally definable sets in NIP}
\author{Martin Bays}
\address{Institut für Mathematische Logik und Grundlagenforschung, Fachbereich
Mathematik und Informatik, Universität Münster, Einsteinstrasse 62, 48149
Münster, Germany}
\email{mbays@sdf.org}

\author{Omer Ben-Neria}
\address{Einstein Institute of Mathematics, Hebrew University of
Jerusalem, 91904, Jerusalem Israel.}
\email{omer.bn@mail.huji.ac.il}

\author{Itay Kaplan}
\address{Einstein Institute of Mathematics, Hebrew University of
Jerusalem, 91904, Jerusalem Israel.}
\email{kaplan@math.huji.ac.il}

\author{Pierre Simon}
\address{Dept.\ of Mathematics, University of California, Berkeley, 970 Evans
Hall \#3840, Berkeley, CA 94720-3840 USA}
\email{simon@math.berkeley.edu}
\thanks{Bays was partially supported by DFG EXC 2044–390685587 and ANR-DFG
AAPG2019 (Geomod).
Ben-Neria was partially supported by the Israel Science Foundation, Grant 1832/19.
Kaplan would like to thank the Israel Science Foundation (ISF) for their
support of this research (grants no. 1254/18 and 804/22).
Simon was partially supported by the NSF (grants no. 1665491 and 1848562).}

\begin{abstract}
We study cofinal systems of finite subsets of $\omega_1$. We show that while 
such systems can be NIP, they cannot be defined in an NIP structure. We deduce 
a positive answer to a question of Chernikov and Simon from 2013: in an NIP 
theory, any uncountable externally definable set contains an infinite 
definable subset. A similar result holds for larger cardinals.
\end{abstract}

\maketitle

\section{Introduction}

Suppose that $M$ is a structure and $x$ a tuple of variables. Recall that a set $X \subseteq M^x$ is \defn{$M$-definable} if there is some formula $\phi(x)$ over $M$ such that $\phi(M)=X$. The set $X$ is \defn{externally definable} if there is some elementary extension $N \succ M$ and a formula $\psi(x)$ over $N$ such that $X = \psi(M) = \set{a\in M^x}{N\vDash \psi(a)}$.

When $\Th(M)$ is stable, all externally definable subsets are in fact $M$-definable (this is a characterization of stability: all types over any model are definable).

Let $T$ be a theory.
We consider the following natural question:
\begin{question} \label{que:hanff number question}
  Is there some (infinite) cardinal $\lambda$ such that for any $M \vDash T$ and any externally definable set $X \subseteq M^k$, if $|X| \geq \lambda$ then $X$ contains an infinite $M$-definable subset?
\end{question}

We cannot hope to say much about externally definable sets in arbitrary theories. In particular, supposing that $T$ has a strong form of IP, the answer to \cref{que:hanff number question} is negative (see \cref{rem:random}).
On the other hand:

\begin{fact}\cite[Corollary 1.12]{CSPairsI}\label{fac:beth}
  Suppose $T$ is NIP. Then the answer to \cref{que:hanff number question} is positive: one can take $\lambda = \beth_\omega$.
\end{fact}

For a complete theory $T$, let $\ext(T)$ be the minimal $\lambda$ as in \cref{que:hanff number question} if such exists, and $\ext(T) = \infty$ otherwise. If $T$ is NIP, \cref{fac:beth} shows that $\ext(T) \leq \beth_\omega$. We first observe that we cannot hope to improve this to $\ext(T) = \aleph_0$.

\begin{example}\cite[just above Question 1.13]{CSPairsI} \label{exa:ext >=alph1}
  Let $M$ be the linear order $(\omega + \Z, <)$, whose theory is NIP (and even dp-minimal, see \cite[Proposition A.2]{simon-NIP}). Then, $\omega$ is externally definable (as is any cut), but no infinite subset of $\omega$ is $M$-definable, since $\Th(M)$ has quantifier elimination after adding the successor and predecessor functions. Thus, $\ext(\Th(M))\geq \aleph_1$.
\end{example}

In their paper \cite[Question 1.13]{CSPairsI}, Chernikov and Simon posed the following question:
\begin{quote}
  Is it true that $\ext(T) \leq \aleph_1$ whenever $T$ is NIP?
\end{quote}

In this paper we positively answer this question (see \cref{mainthm intro}). We use the existence of honest definitions (see \cref{def:honest definitions,fac: uniform of HD}). Let $X = \phi(M,c)$ be externally definable and uncountable and let $\psi(x,z)$ be an honest definition for $\tp_\phiopp(c/M)$. This means that for every finite set $X_0 \subseteq X$, there is some $d\in M^z$ such that
\begin{equation} \label{eqn:HD}
X_0 = \phi(X_0,c) \subseteq \psi(M,d) \subseteq \phi(M,c) = X. \tag{*}
\end{equation}
If one of these sets $Y_d: = \psi(M,d)$ is infinite we are done, so assume for
all $d$ as in \cref{eqn:HD}, $Y_d$ is finite. We get a family of finite subsets of $X$ which is cofinal as a subset of the partial order $\SS^{<\omega}(X)$ of finite subsets of $X$. This raises the question:
\begin{question} \label{que:cofinal family}
  Suppose $\calF$ is a cofinal family of finite subsets of $\aleph_1$. Can $\calF$ have finite VC-dimension?

  In other words, can the relation $\mathord{\in} |_{(\aleph_1 \times \calF)}$ be NIP?
\end{question}

In \cref{thm:cofinal family of VC dim 2} we give a positive answer to \cref{que:cofinal family}. This means that the fact that the honest definition is NIP is not in itself a guarantee that $X$ has an infinite $M$-definable subset (see \cref{rem:honestNIP}).

On the other hand, we prove that if $\calF$ is a cofinal family of finite subsets of $\aleph_1$, then the two-sorted structure $(\aleph_1, \calF;\mathord{\in})$ has IP. We conclude (in \cref{thm:main theorem paper}):
\begin{mainthm}\label{mainthm intro}
  For $T$ NIP, $\ext(T) \leq \aleph_1$.
\end{mainthm}

\subsection{A generalisation to arbitrary cardinals}\label{sec:kappa}
We also consider the following generalisation of \cref{que:hanff number question}.
\begin{question} \label{que:hanff number general question}
  Let $T$ be a theory and $\kappa$ an infinite cardinal. Is there some cardinal $\lambda$ such that for any $M \vDash T$ and any externally definable set $X \subseteq M^k$, if $|X| \geq \lambda$ then $X$ contains an $M$-definable subset of size $\geq \kappa$?
\end{question}

For a complete theory $T$, let $\ext(T,\kappa)$ be the minimal $\lambda$ as in \cref{que:hanff number general question} (if it does not exist, let $\ext(T,\kappa) = \infty$).
So $\ext(T) = \ext(T,\aleph_0)$.
The proof of \cite[Corollary 1.12]{CSPairsI} can easily be adapted to show that if $T$ is NIP then $\ext(T,\kappa) \leq \beth_\omega(\kappa)$.

The following slight adaptation of \cref{exa:ext >=alph1} gives us an NIP theory $T$ (namely DLO) with $\ext(T,\kappa) \geq \kappa^+$ for $\kappa \geq \aleph_1$. Let $I$ be an extension of the linear order $(\kappa,<)$ where between any two ordinals we put a copy of $\Q$. Let $M = \Q + I + \Q$. Then $\M$ is a dense linear order and thus $\Th(\M)$ has quantifier elimination. The set $I$ is externally definable but contains no $M$-definable subset of size $\kappa$.

In fact we prove the main theorem, \cref{thm:main theorem paper}, in this generality: if $T$ is NIP then $\ext(T,\kappa)\leq \kappa^+$.

\subsection{Structure of the paper}
In \cref{sec:Preliminaries}, we give the necessary preliminaries on NIP and honest definitions. In \cref{sec:VC-dim of a cofinal family} we discuss \cref{que:cofinal family}. In \cref{sec:NIP and cofinal families} we prove the technical lemmas needed to prove \cref{thm:main theorem paper}, which is proven in \cref{sec:conclusion} and supplemented by some open questions.

In a previous version of this paper there was a mistake in the proof of \cref{thm:cofinal family of VC dim 2} (pointed out to us by George Peterzil). The old proof involved the construction of well orders of order type $\omega$ on countable ordinals which agree up to finite sets. Since this result may be of independent interest, we put it in \cref{appendix}. 

\subsection{Acknowledgements}
The work on this project started while the third author was a visiting researcher at the 2021 Thematic Program on Trends in Pure and Applied Model Theory at the Fields Institute, and he would like to thank the Fields Institute for their hospitality.

We would like to thank George Peterzil for pointing out an error in the proof of \cref{thm:cofinal family of VC dim 2} that appeared in a previous version. We would also like to thank the anonymous referee for their meticulous reading and useful comments.
\section{Preliminaries} \label{sec:Preliminaries}

\subsection{Notations}\label{subsec:notation}
Our notation is standard. We use $\L$ to denote a first order language and $\phi(x,y)$ to denote a formula $\phi$ with a partition of (perhaps a superset of) its free variables. Let $\phiopp$ be the partitioned formula $\phi(y,x)$ (it is the same formula with the partition reversed).

$T$ will denote a complete theory in $\L$, and $\U \vDash T$ will be a monster model (a sufficiently large saturated model).

When $x$ is a tuple of variables and $A$ is a set contained in some structure (perhaps in a collection of sorts), we write $A^{x}$ to denote the tuples of the sort of $x$ (and of length $|x|$) of elements from $A$; alternatively, one may think of $A^{x}$ as the set of assignments of the variables $x$ to $A$. If $M$ is a structure and $A \subseteq M^x$, $b\in M^y$, then $\phi(A,b)=\set{a\in A}{M \vDash \phi(a,b)}$.

When $B \subseteq \U$, $\L(B)$ is the language $\L$ augmented with constants for elements from $B$ so that a set is $B$-definable if it is definable in $\L(B)$.

For an $\L$-formula $\phi(x,y)$, an \recall{instance} of $\phi$ over $B \subseteq  \U$ is a formula $\phi(x,b)$ where $b \in B^y$, and a (complete) \recall{$\phi$-type} over $B$ is a maximal partial type consisting of instances and negations of instances of $\phi$ over $B$. We write $S_{\phi}(B)$
for the space of $\phi$-types over $B$ in $x$ (in this notation we keep in mind the partition $(x,y)$, and $x$ is the first tuple there). We also use the notation $\phi^1=\phi$ and $\phi^0 = \neg \phi$. For $a \in \U^x$, we write $\tp_\phi(a/B) \in S_\phi(B)$ for its $\phi$-type over $B$.

\subsection{VC-dimension and NIP}\label{subsec:VC-dim and NIP}

\begin{definition} [VC-dimension]
  Let $X$ be a set and $\calF\subseteq\SS(X)$. The pair $(X,\calF)$ is called a \defn{set system}.
  We say that $A\subseteq X$ is \defn{shattered} by $\calF$ if for every $S\subseteq A$ there is $F\in\calF$ such that $F\cap A=S$. A family $\calF$ is said to be a \defn{VC-class} on $X$ if there is some $n<\omega$ such that no subset of $X$ of size $n$ is shattered by $\calF$. In this case the \defn{VC-dimension of $\calF$}, denoted by $\VC(\calF)$, is the smallest integer $n$ such that no subset of $X$ of size $n+1$ is shattered by $\calF$.

  If no such $n$ exists, we write $\VC(\calF)=\infty$.
\end{definition}

\begin{definition} \label{def:NIP formula, theory}
  Suppose $T$ is an $\L$-theory and $\phi(x,y)$ is a formula.
  Say $\phi(x,y)$ is \defn{NIP} if for some/every $M\vDash T$, the family $\set{\phi(M^x,a)}{a\in M^{y}}$ is a VC-class. Otherwise, $\phi$ is \defn{IP} (IP stands for ``Independence Property'' while NIP stands for ``Not IP'').

  The theory $T$ is \defn{NIP} if all formulas are NIP. A structure $M$ is \defn{NIP} if $\Th(M)$ is NIP.
\end{definition}

\begin{definition}\label{def:honest definitions}
  \cite[Definition~3.16 and Remark~3.14]{simon-NIP} Suppose $T$ is an $\L$-theory and $M \vDash T$. Suppose that $\phi(x,y)$ is a formula, $A\subseteq M^{x}$
  is some set and $b \in \U^y$. Say that an $\L$-formula $\psi(x,z)$ (with $z$ a tuple of variables each of the same sort as $x$) 
  is an \defn{honest definition} of $\tp_\phiopp(b/A)$ if for
  every finite $A_0\subseteq A$ there is some $c\in A^{z}$ such
  \[\phi(A_0,b) \subseteq \psi(A,c) \subseteq \phi(A,b).\]
  In other words, for all $a\in A$, if $\psi(a,c)$ holds then so does $\phi(a,b)$
  and for all $a\in A_0$ the other direction holds: if $\phi(a,b)$ holds
  then $\psi(a,c)$ holds.
\end{definition}

The existence of honest definitions for NIP theories was first proved in \cite{CSPairsI}. This was improved in \cite{CS-extDef2} to get \emph{uniformity} of the honest definitions assuming that $T$ is NIP. This was subsequently improved to:

\begin{fact}\cite[Corollary 5.23]{bays2021density} \label{fac: uniform of HD}
  If $\phi(x,y)$ is NIP then there is a formula $\psi(x,z)$ that serves as an honest definition for any $\phiopp$-type over any set $A$ of size $\geq 2$.
\end{fact}

We also recall the Shelah expansion.

\begin{definition}
  For a structure $M$, the \defn{Shelah expansion} $M^{\Sh}$ of $M$ is given by: for any formula $\phi(x,y)$ and any $b \in \U^y$, add a new relation $R_{\phi(x,b)}(x)$ interpreted as $\phi(M,b)$.
\end{definition}

\begin{fact} \cite{MR2570659} \label{fac:Shelah expansion}
  If $T$ is NIP then for any $M\vDash T$, $M^{\Sh}$ is NIP.
\end{fact}

\section{The VC-dimension of cofinal families of finite subsets of an uncountable set} \label{sec:VC-dim of a cofinal family}

The goal of this section is to answer \cref{que:cofinal family}.

\begin{definition}\label{def:cofinal}
  We say that a set $\calF$ of subsets of a set $X$ is \defn{$\omega$-cofinal} if every finite subset of $X$ is contained in some element of $\calF$. (In the case that $\calF$ consists of finite subsets of $X$, we omit ``$\omega$-''.)
\end{definition}

We start with an easy observation.

\begin{remark} \label{rem:stable cofinal}
  If $\calF$ is an $\omega$-cofinal family of subsets of an infinite set $X$ such that if $s \in \calF$ then $|s|<|X|$, then $\in |_{(X \times \calF)}$ is \defn{unstable}: there exist $\sequence{x_i,s_i}{i \in \omega}$ such that $x_i \in s_j$ iff $i\leq j$. Indeed, inductively choose $x_i \in X$ and $s_i \in \calF$ such that $x_i \notin \bigcup_{j<i}s_j$ and $s_i$ contains $\set{x_j}{j\leq i}$.
\end{remark}

The proof of \cite[Corollary 1.12(2)]{CSPairsI} can be adapted to say that if $|X| \geq \beth_\omega$ and $\calF$ is a cofinal family of finite subsets of $X$ then $\calF$ is not a VC-class (i.e., it has IP). In fact, one can also make a connection between the VC-dimension of $\calF$ and the cardinality of $X$ (via the alternation rank of the appropriate relation). The next proposition replaces $\beth_\omega$ with $\aleph_\omega$, and gives a precise lower bound on the VC-dimension in terms of the cardinality of $X$.

\begin{proposition}\label{prop:aleph_omega good enough}
  If $|X| \geq \aleph_n$, then any cofinal system $\calF$ of finite subsets of $X$ has VC-dimension $> n$.
  So any cofinal set system of finite sets on a set of size $\geq \aleph_\omega$ has IP.
\end{proposition}

\begin{proof}
We may assume $X = \aleph_n$.

For finite subsets $A,B \subseteq X$, write
$A \vdash B$ to mean that if $D \in \calF$ contains $A$ then $D \cap B \neq \emptyset$, and write $A \not \vdash B$ for the negation of this.

\begin{observation} \label{obs: does not imply emptset}
  We have $A \not \vdash \emptyset$ for any $A$, since $\calF$ is cofinal.
\end{observation}

\begin{observation} \label{obs:only finite}
  If $A \not \vdash B$, then there are only finitely many $c \in X$ such that
$A \vdash B\cup\{c\}$.
(Indeed, if $D \in \calF$ witnesses $A \not \vdash B$, then we must have $c \in D$.)
\end{observation}

\begin{observation} \label{obs:monotone}
  If $A' \subseteq A$ and $A \not \vdash B$ then $A' \not \vdash B$.
\end{observation}

We find $c_i \in \aleph_i$ for $0\leq i \leq n$ by downwards induction such that for $k = n,...,0,-1$:

\noindent\hypertarget{indImp}{(+)$_k$} $\left\{
\begin{minipage}{0.8\textwidth}
  for any partition $c_{>k} = A \cup B$ where $A,B$ are disjoint and any $b_i \in \aleph_i$ for $i \leq k$,
    \[b_{\leq k} \cup A \not \vdash B.\]
\end{minipage}
\right.$

\hyperlink{indImp}{(+)$_n$} holds by \cref{obs: does not imply emptset}.

\hyperlink{indImp}{(+)$_{-1}$} means that $\set{c_i}{i<n+1}$ is shattered, from which we conclude.

Suppose \hyperlink{indImp}{(+)$_k$} holds, we choose $c_k \in \aleph_k$ such that (+)$_{k-1}$ holds.

Such a $c_k$ exists because for each of the $\aleph_{k-1}$ choices for $b_{<k}$ and $A$, there are only finitely many choices to rule out. More explicitly, for every choice of $b_{<k}$ as above, and any $A \subseteq c_{>k}$ such that $b_{<k} \cup A \not \vdash c_{>k} \setminus A$, let $s_{b_{<k},A} = \set{c\in \aleph_k}{b_{<k} \cup A \vdash c_{>k} \setminus A \cup \{c\}}$. By \cref{obs:only finite}, $s_{b_{<k},A}$ is finite for each such $b_{<k},A$, and let $c_k \in \aleph_k \setminus (\bigcup\set{s_{b_{<k},A}}{b_{<k} \cup A \not \vdash c_{>k} \setminus A} \cup c_{>k})$.

\hyperlink{indImp}{(+)$_{k-1}$} holds: if $c_k \in A$ then we are done by induction, and otherwise $c_k \in B$ and this follows from \cref{obs:monotone} and the choice of $c_k$.
\end{proof}

\begin{remark} \label{rem:aleph_alpha+omega}
  With the same proof \emph{mutatis mutandis} one can see that if $\calF$ is an $\omega$-cofinal family of subsets of $X$, each of size $<\aleph_\alpha$, and if $|X| \geq \aleph_{\alpha+n}$, then $\calF$ has VC-dimension $>n$.
\end{remark}

\begin{theorem} \label{thm:cofinal family of VC dim 2}
  There is a cofinal family $\calF$ of finite subsets of $\aleph_1$ of VC-dimension 2.
\end{theorem}
\begin{proof}
  Let $\delta\leq \omega_1$. Suppose that $\mathcal{C} = \sequence{\mathord{<^\alpha}}{\alpha<\delta}$ is a sequence of linear orders, where $<^\alpha$ is a linear order on $\alpha$. We define the following relation on triples $\alpha,\beta,\gamma<\delta$: $\alpha, \beta \vdash_{\mathcal{C}} \gamma$ iff 
  $\beta,\gamma<\alpha$ and $\gamma<^\alpha \beta$.
  Say $B \subseteq \delta$
  is \defn{$\vdash_{\mathcal{C}}$-closed} if for any $\alpha,\beta\in B$, if $\alpha,\beta 
  \vdash_{\mathcal{C}} \gamma$ then $\gamma\in B$.

  We inductively define well-orders $<^\alpha$ on $\alpha < \omega_1$ such that
  \\ \noindent\hypertarget{cof}{$(*)_\alpha$} 
  any finite subset $A \subseteq \alpha$ extends to a finite 
  subset $A \subseteq B \subseteq \alpha$ such that $B\cup\{\alpha\}$ is $\vdash_{\mathcal{C}_\alpha}$-closed for $\mathcal{C}_{\alpha}:=\sequence{<^\beta}{\beta\leq \alpha}$.

  \hyperlink{cof}{$(*)_0$} holds with $<^0$ the empty order.

  Suppose \hyperlink{cof}{$(*)_\alpha$} holds.
  Let $<^{\alpha+1}$ be the order obtained from $<^\alpha$ by putting $\alpha$ 
  at the start: $\mathord{<^{\alpha+1}} = \mathord{<^\alpha} \cup {\set{(\alpha,\beta)}{\beta<\alpha}}$. 
  Let $A \subseteq \alpha+1$.
  By \hyperlink{cof}{$(*)_\alpha$}, let $B \subseteq \alpha$ be a finite set containing $A \setminus 
  \{\alpha\} \subseteq \alpha$ such that $B' := B \cup \{\alpha\}$ is 
  $\vdash_{\mathcal{C}_{\alpha}}$-closed.
  Then it follows from the definition of $<^{\alpha+1}$ that also $B' \cup 
  \{\alpha+1\}$ is $\vdash_{\mathcal{C}_{\alpha+1}}$-closed. Since $B'$ is finite and contains $A$, we 
  conclude that \hyperlink{cof}{$(*)_{\alpha+1}$} holds.

  Suppose that $\eta < \omega_1$ is a limit ordinal and \hyperlink{cof}{$(*)_\alpha$} holds for all $\alpha < \eta$. Note that for $\alpha<\beta<\eta$ and any $B\subseteq \alpha$, $B$ is $\vdash_{\mathcal{C}_\alpha}$-closed iff $B$ is $\vdash_{\mathcal{C}_\beta}$-closed. 

  Since $\eta$ is countable, it follows that $\eta = \bigcup_{n \in \omega} 
  S_n$ where for each $n<\omega$, $S_n$ is finite, $S_n \subseteq S_{n+1}$ and $S_n$ is $\vdash_{\mathcal{C}_\alpha}$-closed for any (some) $\alpha<\eta$ such that $S_n \subseteq \alpha$. (In the construction, given $S_n$, let $S_n' = S_n \cup \{\beta_n\}$ where $\sequence{\beta_n}{n<\omega}$ enumerates $\eta$ and let $S_{n+1}$ be finite and $\vdash_{\mathcal{C}_{\alpha}}$-closed containing $S_n'$ for $\alpha<\eta$ such that $S_n'\subseteq \alpha$.) We define $<^\eta$ to be of order 
  type $\omega$ in such a way that each $S_n$ is an initial segment.
  Then $(*)_\eta$ holds: if $A$ is a finite subset of $\eta$, then $A$ is 
  contained in some $S_n$ which is finite and $\vdash_{\mathcal{C}_\alpha}$-closed for any $\alpha$ large enough, and since 
  $S_n$ is an initial segment of $<^\eta$, $S_n\cup\{\eta\}$ is 
  $\vdash_{\mathcal{C}_\eta}$-closed.

  Finally, let $\mathcal{C} = \sequence{\mathord{<^\alpha}}{\alpha<\omega_1}$ and $\mathord{\vdash} = \mathord{\vdash}_{\mathcal{C}}$. Let $\calF$ be the family of finite subsets of $\omega_1$ which are $\vdash$-closed. By the above construction, $\calF$ is cofinal. As for any triple 
  $\alpha_0,\alpha_1,\alpha_2<\omega_1$ of distinct ordinals there is some 
  permutation $\sigma$ of 3 such that 
  $\alpha_{\sigma(0)},\alpha_{\sigma(1)}\vdash \alpha_{\sigma(2)}$, $\calF$ 
  does not shatter any set of size 3.
\end{proof}

\begin{corollary}
  The following statement is independent of ZFC: there is an NIP cofinal family of finite subsets of $2^{\aleph_0}$.
\end{corollary}
\begin{proof}
  On the one hand CH is consistent with ZFC (by G\"odel's theorem, see e.g., \cite[Theorem 13.20]{jech}), and on the other hand it is consistent with ZFC that $\aleph_\omega < 2^{\aleph_0}$ (using Cohen forcing, see e.g., \cite[Chapter 15, ``Cohen Reals'']{jech}). Thus, the statement follows from \cref{prop:aleph_omega good enough,thm:cofinal family of VC dim 2}.
\end{proof}

\begin{question}
  Is there a cofinal family of finite subsets of $\aleph_2$ of VC-dimension 3? More generally: is the bound in \cref{prop:aleph_omega good enough} tight, or can we improve $\aleph_\omega$ to a smaller cardinal?
\end{question}

\section{NIP and cofinal families of finite subsets of an uncountable set}\label{sec:NIP and cofinal families}
This section is devoted to proving the following theorem.

\begin{theorem}\label{thm:cofinal family implies IP}
  Suppose that $\kappa$ is an infinite cardinal, $|X| \geq \kappa^+$, and $\calF$ is an $\omega$-cofinal family of subsets of $X$, each of size $<\kappa$. Then $(X,\calF;\mathord{\in})$ has IP (as a two-sorted structure whose only relation is $\mathord{\in} \subseteq X \times \calF$).
\end{theorem}

The proof relies upon the following lemma.

\begin{lemma} \label{lem:the technical lemma}
  Let $\kappa$ be any infinite cardinal. Assume that:
  \begin{enumerate}
    \item $|X|\geq \kappa^+$.
    \item $R \subseteq X^n$ and $1\leq n$.
    \item For every $a_1,\dots,a_{n-1}\in X$, $|\set{a_0 \in X}{R(a_0,a_1,\dots,a_{n-1})}|<\kappa$.
    \item For every set $A \subseteq X$ of size $|A|=n$, for some $a\in A$ and some tuple $\bar{a}\in (A\setminus{a})^{n-1}$, $R(a,\bar{a})$ holds.
  \end{enumerate}
  Then, there is some partition of $\{1,\dots,n-1\}$ into nonempty disjoint sets $u,v$ such that, letting $x: = \sequence{x_i}{i\in u\cup \{0\}}$ and $y := \sequence{x_i}{i\in v}$, the partitioned formula $\phi(x,y):= R(x_{0},x_1, \dots, x_{n-1})$ has IP.
\end{lemma}

\begin{remark}
  \cref{lem:the technical lemma} does not hold if we replace $\kappa^+$ by $\kappa$ in (1). Indeed, let $X = \omega$ and let $R(x,y) = (x<y)$. Then (2)--(4) hold for $\kappa = \aleph_0$ and $n=2$ but $R$ is NIP (by e.g., \cite[Proposition A.2]{simon-NIP}).  
\end{remark}

\begin{remark}\label{rem:n=2}
  Note that conditions (1)--(4) imply that $n>2$. If $n=1$ then by (3), $R$ defines a set of size $<\kappa$, but by (4), $R$ contains $X$, contradicting (1).
  Suppose that $n=2$ and for $a \in X$ let $s_a = \set{b \in X}{R(b,a)}$.
  Let $X_0, X_1 \subseteq X$ be such that $X_0 \cap X_1 =\emptyset$, $|X_0|=\kappa$ and $|X_1|=\kappa^+$. Let $S = \bigcup\set{s_a}{a\in X_0}$. As $|S|\leq \kappa$, there must be some $b\in X_1\setminus S$. As $|s_b|<\kappa$, there must be some $a \in X_0 \setminus s_b$. Then $a \notin s_b$ and $b \notin s_a$, contradicting (4).
\end{remark}

The following example shows that the conditions of \cref{lem:the technical lemma} can hold when $n=3$.
\begin{example}\label{exa:Omer's example}
  Suppose that for each $\alpha<\omega_1$, $<^\alpha$ is a well order on $\alpha$ of order type $\omega$. For $\alpha,\beta,\gamma <\omega_1$, let $R(\gamma, \beta, \alpha)$ hold iff $\gamma,\beta <\alpha$ and $\gamma <^\alpha \beta$. Then $R$ satisfies the conditions of \cref{lem:the technical lemma} with $\kappa = \aleph_0$.
\end{example}

\begin{remark}
  In essence, the proof of \cref{lem:the technical lemma} is an induction on $n$, with \cref{rem:n=2} as the base case. However, we need to keep track of sets witnessing IP ($D^{k,j,\bar{c}}_\barA$ in the proof below), which substantially complicates the proof.
\end{remark}

\begin{proof}[Proof of \cref{lem:the technical lemma}]
  Assume not, i.e., that
  \begin{enumerate}
    \item[(5)] for any partition of $\{1,\dots,n-1\}$ into nonempty disjoint sets $u,v$, letting $x: = \sequence{x_i}{i\in u\cup \{0\}}$ and $y := \sequence{x_i}{i\in v}$ the partitioned formula $\phi(x,y):= R(x_{0},x_1, \dots, x_{n-1})$ is NIP.
  \end{enumerate}

  Define $R' \subseteq X^n$ by $R'(a_0,\dots,a_{n-1})$ iff for some tuple $\bar{a} \in \{a_1,\dots,a_{n-1}\}^{n-1}$, $R(a_0,\bar{a})$ holds.
  Note that $R'$ satisfies (2)--(5) above (it satisfies (5) as a finite disjunction of NIP relations; in fact (5) can now be simplified by saying that $R'(x_{0},\dots,x_{k-1};x_{k},\dots,x_{n-1})$ is NIP for any $1<k<n$). Thus, we can replace $R$ with $R'$ and assume in addition that
  \begin{enumerate}
    \item[(6)] For any tuple $\bar{a} \in X^{n-1}$ and for any permutation $\bar{a}'$ of $\bar{a}$, $R(X,\bar{a})=R(X,\bar{a}')$.
  \end{enumerate}

  For any nonempty set $t \subseteq X$ of size $\leq n-1$, let $s_t = R(X,\bar{a})$ where $\bar{a}$ is any enumeration of $t$ of length $n-1$; 
  this is well-defined by (6). We can then restate (4) as:
  \item [$\boxtimes$] For every $t$ of size $n$, for some $a\in t$, $a \in s_{t \setminus \{a\}}$.

  We may assume that $|X|=\kappa^+$, and even that $X = \kappa^+$. When we say that a subset of $X$ is ``cofinal'' or ``contains an end segment of some cofinal set'', we mean with respect to the canonical order on $\kappa^+$. For a cofinal set $D \subseteq X$ and some property $P \subseteq X$, write $\forall^* x\in D \ P(x)$ to mean that $P$ contains an end segment of $D$.
  By downwards induction on $k \in [2,n]$ (note that by \cref{rem:n=2}, $n>2$, so this range for $k$ makes sense), we will find:
  \begin{enumerate}[(A)]
    \item $m_k < \omega$, and
    \item a cofinal set $D^{k,j,\bar{c}}_\barA \subseteq X$
      for every $j \in [k,n)$ and $\bar{c}\in X^{j-k}$ and $\barA \in \prod_{i\in [k,j]}{\SS(m_i)}$, 
  \end{enumerate}
  such that $\boxplus_k$ and $\boxtimes_k$ below hold. To state these conditions, we first introduce some additional notation:

\begin{itemize}
  \item For $l,j$ with $k \leq l \leq j \leq n$, let $M_l^j := \prod_{i \in [l,j]}{\SS(m_i)}$. We denote elements of $M_l^j$ by $(j+1-l)$-tuples $\barA$. For $j<l$, set $M_l^j := \{\emptyset\}$.

  \item For $j \in [k,n]$, $t\subseteq X$ of size $k-1$, $\bar{c}\in X^{j-k}$, and $\barA \in M_k^{j-1}$, define sets $s^{k,j,\bar{c},\barA}_t$ as follows. We set
       $s^{k,n,\bar{c},\barA}_t =s_{t \cup \bar{c}}$, and then define recursively for $j \in [k,n)$:
        \[ s^{k,j,\bar{c},\barA}_t=\set{a\in X}{\exists A \subseteq m_j \ \forall^* c\in D^{k,j,\bar{c}}_{\barA A} \ a\in s^{k,j+1,\bar{c}c,\barA A}_t } .\]

  \item For $t\subseteq X$ of size $k-1$, let $s^k_t = s^{k,k,\emptyset,\emptyset}_t$.
\end{itemize}

  Now we can state the conditions to be satisfied by our inductive construction:

  \begin{itemize}
    \item [$\boxplus_k$] For every $j \in (k,n)$ and every $c\bar{c}\in X^{j-k}$ and $A\barA \in M_k^j$, $D^{k,j,c\bar{c}}_{A\barA} \subseteq D^{k+1,j,\bar{c}}_{\barA}$ (for $k\geq n-1$, this condition holds trivially). 
    \item [$\boxtimes_k$]
      For all $t\subseteq X$ of size $|t|=k$, for some $a \in t$, $a \in s^k_{t \setminus \{a\}}$.
  \end{itemize}

  Note that each $|s^{k,j,\bar{c},\barA}_t|<\kappa$ by downwards induction on $j$: for $j=n$ this is clear, and suppose that $|s^{k,j+1,\bar{c}c,\barA A }_t|<\kappa$ for all $c \in X$ and $A\subseteq m_j$. Towards a contradiction, assume that $s^{k,j,\bar{c},\barA}_t$ contains a set $F$ of size $\kappa$. We may assume that for some $A \subseteq m_j$ and all $a \in F$, $\forall^* c\in D^{k,j,\bar{c}}_{\barA A} \ a\in s^{k,j+1,\bar{c}c,\barA A}_t$. For any $a \in F$ there is an end segment $F_a$ of $D^{k,j,\bar{c}}_{\barA A}$ such that for any $c \in F_a$, $a\in s^{k,j+1,\bar{c}c,\barA A}_t$. Since $D^{k,j,\bar{c}}_{\barA A}$ is cofinal in $\kappa^+$ (which is a regular cardinal), $\bigcap_{a\in F}F_a$ contains an end segment of $D^{k,j,\bar{c}}_{\barA A}$, and in particular is nonempty. Let $c \in \bigcap_{a\in F}F_a$. Then $F \subseteq s^{k,j+1,\bar{c}c,\barA A}_t$, contradicting the induction hypothesis.

  Note also that for $t \subseteq X$ of size $n-1$, $s_t = s^{n}_t$.

  We now proceed with the inductive construction of the $m_k$ and $D^{k,j,\bar{c}}_\barA$.

  For $k=n$, let $m_n=0$. 
  Then $\boxplus_n$ holds trivially and $\boxtimes_n$ holds by $\boxtimes$ above.

  Assume that $2\leq k<n$ and we found $m_{k'}$ and cofinal sets $D^{k',j,\bar{c}}_\barA$ such that $\boxplus_{k'}$ and $\boxtimes_{k'}$ hold for all $k'>k$. We want to find $m_k$ and sets $D^{k,j,\bar{c}}_\barA$ such that $\boxplus_k$ and $\boxtimes_k$ hold.

  For $m<\omega$ we let $\otimes_m$ be the following statement: there are
  \begin{itemize}
    \item cofinal sets $D^{k,j,\bar{c}}_\barA$ for $j \in [k,n)$ and $\bar{c}\in X^{j-k}$ and $\barA \in (\SS(m) \times M_{k+1}^j)$, 
    and
    \item subsets $t_i \subseteq X$ of size $k$ for $i < m$,
  \end{itemize}
  such that:
  \begin{enumerate}[(I)]
    \item $\boxplus_k$ holds with $m$ playing the role of $m_k$;
    \item if $B \neq A$ are subsets of $m$, and $i := \min (B \symdif A) \in B$, then for some enumeration $\bar{a}$ of $t_i$, the following hold:
    \begin{itemize}
      \item [$\oplus$] for all $c_k \in D^{k,k,\emptyset}_{(B)}$, for some $A_{k+1} \subseteq m_{k+1}$ and all $c_{k+1} \in D^{k,k+1,(c_k)}_{(B,A_{k+1})}$, \dots, for some $A_{n-1} \subseteq m_{n-1}$ and all $c_{n-1}$ in $D^{k,n-1,(c_k,\dots,c_{n-2})}_{(B,\dots,A_{n-1})}$, $$R(\bar{a},c_k,\dots,c_{n-1}) ;$$
      \item [$\ominus$] for all $c_k \in D^{k,k,\emptyset}_{(A)}$, for all $A_{k+1} \subseteq m_{k+1}$ and all $c_{k+1} \in D_{(A,A_{k+1})}^{k,k+1,(c_k)}$, \dots, for all $A_{n-1} \subseteq m_{n-1}$ and all $c_{n-1}$ in $D_{(A,\dots,A_{n-1})}^{k,n-1,(c_k,\dots,c_{n-2})}$, $$\neg R(\bar{a},c_k,\dots,c_{n-1}) .$$
    \end{itemize}
  \end{enumerate}

  If $\otimes_m$ holds for all $m<\omega$, we get IP as we now explain. Consider the formula $$\phi(\bar{x},\bar{y}) = \bigvee_{\barA \in M_{k+1}^{n-1}}R(x_0,\dots,x_{k-1};y_k,y_{k+1}^\barA,y_{k+2}^\barA,\dots,y_{n-1}^\barA).$$
  Fix some $m<\omega$. For $A\subseteq m$, we define a $\bar{y}$-tuple $\bar{c}^A$ as follows. Let $c^A_k \in D^{k,k,\emptyset}_{(A)}$ and for $j \in [k+1,n)$ and $(A_{k+1},...,A_{n-1}) \in M_{k+1}^{n-1}$, inductively let $c^{A,\barA}_j \in D^{k,j,(c^A_k,c^{A,\barA}_{k+1},\dots,c^{A,\barA}_{j-1})}_{(A,A_{k+1},...,A_j)}$. Then, by (II) we get that if $B \neq A$ and $i =\min (B\symdif A) \in B$, then for some tuple $\bar{a}$ enumerating $t_i$, $\phi(\bar{a},\bar{c}^B)$ holds, while $\phi(\bar{a},\bar{c}^A)$ does not. Let $E$ be the set of all $\bar{x}$-tuples $\bar{a}$ enumerating $t_i$ for all $i<m$. We get that the number of $\phi$-types in $\bar{y}$ over $E$ is exponential in $m$ (at least $2^m$). However, $|E|\leq m k!$. By Sauer-Shelah (\cite[Lemma 6.4]{simon-NIP}), we get that $\phi(\bar{x},\bar{y})$ has IP. As NIP formulas are closed under Boolean combinations, we get that $R(x_0,\dots,x_{k-1};y_k,\dots,y_{n-1})$ has IP, contradicting (5).

  We first show that $\otimes_0$ holds. Let $D_{(\emptyset)} ^{k,k,\emptyset}=X$ and for $j>k$, $\barA \in M_{k+1}^j$, $\bar{c} \in X^{j-(k+1)}$ and $c\in X$, let $D_{\emptyset\barA} ^{k,j,c\bar{c}} = D_\barA ^{k+1,j,\bar{c}}$. Then (I) is immediate, and (II) is trivially satisfied.

  Let $m_k$ be maximal such that $\otimes_{m_k}$ holds, witnessed by $D^{k,j,\bar{c}}_\barA$ and $t_i$. We claim that this $m_k$ and $D^{k,j,\bar{c}}_\barA$ satisfy $\boxplus_k$ and $\boxtimes_k$.
  $\boxplus_k$ is satisfied by (I), so we are left to check $\boxtimes_k$.

  Assume that $\boxtimes_k$ does not hold. Then there is some $t \subseteq X$ of size $k$ witnessing this: for all $a\in t$, $a \notin s^k_{t\setminus \{a\}}$.  We will show that letting $t_{m_k}:=t$, we can find new $D^{k,j,\bar{c}}_\barA$ for $\barA \in \SS(m_k+1) \times M_{k+1}^j$ and $\bar{c}\in X^{j-k}$ witnessing $\otimes_{m_k+1}$. We will construct two sequences of cofinal sets, $E_\barA^{k,j,\bar{c}}$ and ${F_\barA^{k,j,\bar{c}}}$, that will then be used to find suitable $D$'s.


  Let $A_k \subseteq m_k$. 
  Let $E_{(A_k)}^{k,k,\emptyset}=D_{(A_k)}^{k,k,\emptyset} \setminus (s^{k+1}_t \cup t)$. Since $s^{k+1}_t$ has size $<\kappa$, $E_{(A_k)}^{k,k,\emptyset}$ is still cofinal. Let $c_k \in E_{(A_k)}^{k,k,\emptyset}$. By $\boxtimes_{k+1}$ applied to $t \cup \{c_k\}$, and as $c_k \notin s^{k+1}_t$, for some $a_{c_k,A_k}  \in t$, we have $a_{c_k,A_k} \in s^{k+1}_{\{c_k\}\cup t \setminus \{a_{c_k,A_k}\}}$. As $t$ is finite, by reducing $E_{(A_k)}^{k,k,\emptyset}$, we may assume that there is some $a_{A_k} \in t$ such that $a_{A_k} \in s^{k+1}_{\{c_k\}\cup t \setminus \{a_{A_k}\}}$ for any $c_k \in E_{(A_k)}^{k,k,\emptyset}$.

  Let $j \in (k,n)$, $\barA \in M_k^j$, and $\bar{c}\in X^{j-k}$. Write $\bar{c}=c_k\bar{c}'$ and $\barA = A_k\barA'$, so $\barA' \in M_{k+1}^j$. We define cofinal sets $E_{\barA}^{k,j,\bar{c}}$ as follows 

  \begin{itemize}
    \item If 
      $\forall^* c \in D_{\barA'}^{k+1,j,\bar{c}'} \ a_{A_k} \in s_{\{c_k\}\cup t \setminus \{a_{A_k}\}}^{k+1,j+1,\bar{c}'c,\barA'}$ then let $S\subseteq D_{\barA'}^{k+1,j,\bar{c}'}$ be an end segment witnessing this, and set $E_{\barA}^{k,j,\bar{c}}= S \cap D_{\barA}^{k,j,\bar{c}}$. Note that $E_{\barA}^{k,j,\bar{c}}$ is cofinal as $D_{\barA}^{k,j,\bar{c}} \subseteq D_{\barA'}^{k+1,j,\bar{c}'}$ by $\boxplus_k$.
    \item Otherwise, let $E_{\barA}^{k,j,\bar{c}} = D_{\barA}^{k,j,\bar{c}}$.
  \end{itemize}

  By (upwards) induction on $j \in [k,n)$ one proves that: 
  \begin{itemize}
    \item [($\dagger_j$)] For any $A_k\subseteq m_k$ and any $c_k \in E_{(A_k)}^{k,k,\emptyset}$ there is some $A_{k+1} \subseteq m_{k+1}$ such that for any $c_{k+1} \in E_{(A_k,A_{k+1})}^{k,k+1,(c_k)}$ there is some $A_{k+2} \subseteq m_{k+2}$ such that \dots for any $c_j \in E_{(A_k,\dots,A_j)}^{k,j,(c_k,\dots,c_{j-1})}$, $a_{A_k} \in s_{\{c_k\}\cup t \setminus \{a_{A_k}\}}^{k+1,j+1,(c_{k+1},\dots,c_j),(A_{k+1},\dots,A_j)}$.
  \end{itemize}

  Now for $A_k \subseteq m_k$, let $F_{(A_k)}^{k,k,\emptyset} \subseteq D_{(A_k)}^{k,k,\emptyset}$ be a cofinal set such that $a_{A_k} \notin s^{k,k+1,(c),(A_k)}_{t\setminus \{a_{A_k}\}}$ for any $c \in F_{(A_k)}^{k,k,\emptyset}$; such a set exists since $a_{A_k} \notin s^k_{t\setminus \{a_{A_k}\}} = s^{k,k,\emptyset,\emptyset}_{t\setminus \{a_{A_k}\}}$.

  Then for any $j \in (k,n)$, any $\bar{c} \in X^{j-k}$ and any $\barA = (A_k,...,A_j) \in M_k^j$, we similarly define cofinal sets $F_{\barA}^{k,j,\bar{c}}$ as follows.
  \begin{itemize}
    \item If $a_{A_k} \in s^{k,j,\bar{c},(A_k,...,A_{j-1})}_{t\setminus \{a_{A_k}\}}$, let $F_\barA^{k,j,\bar{c}}=D_\barA^{k,j,\bar{c}}$, and
    \item if $a_{A_k} \notin s^{k,j,\bar{c},(A_k,...,A_{j-1})}_{t\setminus \{a_{A_k}\}}$, let $F_\barA^{k,j,\bar{c}}\subseteq D_\barA^{k,j,\bar{c}}$ be cofinal such that $a_{A_k} \notin s^{k,j+1,\bar{c}c,\barA}_{t\setminus \{a_{A_k}\}}$ for any $c \in F_\barA^{k,j,\bar{c}}$.
  \end{itemize}


  Recall that by choice of $t$, $a \notin s^k_{t\setminus \{a\}}$ for all $a\in t$. By (upwards) induction on $j \in [k,n]$ one proves that: 
  \begin{itemize}
    \item[($\$_j $)] If $(A_k,...,A_{j-1}) \in M_k^{j-1}$ and $(c_k,\dots,c_{j-1}) \in X^{j-k}$ are such that $c_i \in F_{(A_k,\dots,A_i)}^{k,i,(c_k,\dots,c_{i-1})}$ for every $i \in [k,j)$, then for every $i \in [k,j]$,
      $$a_{A_k} \notin s^{k,i,(c_k,\dots,c_{i-1}),(A_k,\dots,A_{i-1})}_{t\setminus \{a_{A_k}\}} .$$
  \end{itemize}


  Now, for any $A\barA\in \SS(m_k+1)\times M_{k+1}^j$, let $G_{A\barA}^{k,j,\bar{c}} := F_{A\barA}^{k,j,\bar{c}}$ if $m_k \notin A$, else let $G_{A\barA}^{k,j,\bar{c}} := E_{(A \cap m_k)\barA}^{k,j,\bar{c}}$. 
  Now we show that $\sequence{t_i}{i\leq m_k}$ and these $G_{A\barA}^{k,j,\bar{c}}$ witness $\otimes_{m_k+1}$.
  Note that $G_{A\barA}^{k,j,\bar{c}}$ is a cofinal subset of $D_{(A\cap
  m_k)\barA}^{k,j,\bar{c}}$. 
  Hence $\boxplus_k$ still holds, establishing (I).
  For (II), let $B \neq A$ be subsets of $m_k+1$. If $i=\min( B \symdif A) \in B$ and $i < m_k$, then $i=\min({(B\cap m_k) \symdif (A\cap m_k)})$ and so $\oplus$ and $\ominus$ still hold (using $G_{A\barA}^{k,j,\bar{c}} \subseteq D_{(A\cap m_k)\barA}^{k,j,\bar{c}}$). If not, then $B = A\cup \{m_k\}$. Let $\bar{a}$ be an enumeration of $t_{m_k}$ starting with $a_{A}$. Then $\oplus$ follows from $(\dagger_{n-1})$ and $\ominus$ follows from $(\$_n)$, as required.

  This completes the construction of $m_k$ and $D_\barA^{k,j,\bar{c}}$ as in (A), (B) above for all $k \in [2,n]$.

  Finally, $\boxtimes_2$ yields a contradiction by the argument of \cref{rem:n=2}.
  Indeed, by $\boxtimes_2$ we get that for any distinct $a,b \in X$, either $a \in s^2_{\{b\}}$ or $b \in s^2_{\{a\}}$. Let $X_0, X_1 \subseteq X$ be such that $X_0 \cap X_1 =\emptyset$, $|X_0|=\kappa$ and $|X_1|=\kappa^+$. Let $S = \bigcup\set{s^2_{\{a\}}}{a\in X_0}$. As $|S|\leq \kappa$, there must be some $b\in X_1\setminus S$. As $|s^2_{\{b\}}|<\kappa$, there must be some $a \in X_0 \setminus s^2_{\{b\}}$. But then $a \notin s^2_{\{b\}}$ and $b \notin s^2_{\{a\}}$ --- contradiction.
\end{proof}

\begin{proof}[Proof of \cref{thm:cofinal family implies IP}]
  Suppose that $|X| \geq \kappa^+$ and that $\calF$ is a cofinal family of subsets of $X$, each of size $<\kappa$. Suppose that $\VC(\calF) = n$. 

  For any $0\leq k \leq n$ and any $m \leq k$, let $R_{m,k}(x_0,\dots, x_k)$ be the relation defined by:

  \[ [\exists t\in \calF\; \bigwedge_{1\leq i \leq k} (x_i \in t)^{(i \leq m)}] \land
   [\forall t \in \calF\; ((\bigwedge_{1\leq i \leq k} (x_i \in t)^{(i \leq m)})\to x_0 \in t)].\]

  (If $k=0$ the conjunction is empty and thus holds trivially, meaning that $R_{0,0}(x_0) = \forall t\in \calF\; x_0\in t$.)

  Let $R(x_0,x_1,\dots,x_n) = \bigvee_{m \leq k\leq n} R_{m,k}(x_0,\dots,x_k)$. We claim that $R$ satisfies the conditions of \cref{lem:the technical lemma} on $X$.
  Conditions (1) and (2) are trivial, condition (3) follows from the fact that it is true for each $m,k$ separately and that each $t \in \calF$ has size $<\kappa$ (using the existential clause of the definition of $R_{m,k}$).

  We show condition (4). Suppose that $A \subseteq X$ has size $n+1$. Since $\VC(\calF) = n$, $\calF$ does not shatter $A$. Let $B \subseteq A$ be of minimal size such that $\calF$ does not shatter $B$. Note that $B$ is nonempty and let $k = |B|-1$. 
  Since $B$ is not shattered, there is some $B_0 \subseteq B$ such that for no $t \in \calF$, $t\cap B = B_0$. Note that $B_0 \neq B$ since $\calF$ is $\omega$-cofinal (and $B$ is finite). Let $m = |B_0|$. Let $a_0\in B \setminus B_0$ and let $a_1,\dots,a_k$ enumerate $B \setminus \{a_0\}$ such that $a_i \in B_0$ iff $i \leq m$. It follows that $R_{m,k}(a_0,a_1,\dots,a_k)$ holds: the first clause holds by the minimality of $B$ (any proper subset is shattered), and the second clause follows by the choice of $B_0$.

  By \cref{lem:the technical lemma}, for some permutation $\sigma$ of $\{1,\dots,n\}$ and some $1<k<n$ the partitioned formula $R(x_{0},x_{\sigma(1)}, \dots, x_{\sigma(k-1)}; x_{\sigma(k)},\dots ,x_{\sigma(n-1)})$ has IP and we are done.
\end{proof}

\begin{question}
  Let $R(x, y, z)$ be the relation from \cref{exa:Omer's example}. The proof of \cref{lem:the technical lemma} yields that $R(x,y;z)$ has IP.  Could the relation $R(x;y,z)$ be NIP? Note that $\set{R(\omega_1;\beta,\alpha)}{\beta,\alpha}$ is a cofinal family of finite subsets of $\omega_1$ (see \cref{thm:cofinal family of VC dim 2}).

  Similarly, we do not know whether the formula $\phi(x,z;y) = R(x,y,z)$ has IP. 
\end{question}

\section{Conclusion and final thoughts}\label{sec:conclusion}

We conclude with the final theorem, i.e., the generalisation of \cref{mainthm intro} discussed in \cref{sec:kappa}.
\begin{theorem}\label{thm:main theorem paper}
  Let $\kappa$ be an infinite cardinal. If $T = \Th(M)$ is NIP and $X \subseteq M^k$ is externally definable of size $\geq \kappa^+$, then $X$ contains an $M$-definable subset of size $\geq \kappa$. In other words,  $\ext(T,\kappa) \leq \kappa^+$.
\end{theorem}

\begin{proof}
  Suppose that $X$ is defined by $\phi(x,c)$ for some formula $\phi(x,y)$. Let $\psi(x,z)$ be an honest definition for $\tp_\phiopp(c/M)$.
  This means that for every finite set $X_0 \subseteq X$, there is some $d\in M^z$ such that

  \[X_0 = \phi(X_0,c) \subseteq \psi(M,d) \subseteq \phi(M,c) = X.\]

  Let $Y = \set{d\in M^z}{\psi(M,d) \subseteq X}$. Note that $Y$ is definable in $M^{Sh}$. If for some $d\in Y$, $\psi(M,d)$ has size $\geq \kappa$, we are done, so assume for all $d\in Y$, $|\psi(M,d)|<\kappa$.  Let $\calF = \set{\psi(M,d)}{d \in Y}$. Then $N:=(X,\calF; \in)$ is interpretable in $M^{Sh}$. By \cref{fac:Shelah expansion} it follows that $N$ is NIP, contradicting \cref{thm:cofinal family implies IP} as required.
\end{proof}

\begin{remark}\label{rem:honestNIP}
Note that the above proof implies that in an NIP theory, if $X = \phi(M,c)$ is externally definable of size $\kappa^+$, then an instance of the honest definition of $\tp_\phiopp(c/M)$ has size $\geq \kappa$. By \cref{fac: uniform of HD}, we know that the existence of an honest definition $\psi(x,z)$ only requires $\phi$ to be NIP. We do not know if $\psi$ itself can be chosen to be NIP (this is open even for the finite case, see \cite[Question 22]{UDTFS}). However, even if it were NIP, we cannot get a contradiction as in the proof above due to \cref{thm:cofinal family of VC dim 2}.
\end{remark}

\begin{question}\label{que:local ext def}
  Suppose that $M$ is a structure and $X = \phi(M,c)$ is externally definable of size $\geq \aleph_1$. Suppose that $\phi$ is NIP. Does it follow that $X$ contains an infinite definable subset?
\end{question}

Note that when $T$ eliminates the quantifier $\exists^\infty$, the answer is ``yes'' (even just assuming that $X$ is infinite), as in \cite[Corollary 1.12(1)]{CSPairsI}: by \cref{fac: uniform of HD} (or \cite[Theorem~3.13]{simon-NIP}), $\tp_{\phiopp}(c/M)$ has an honest definition $\psi$, and so as above some instance $\psi(M,d) \subseteq X$ contains a finite subset $X_0$ large enough that, by elimination of $\exists^\infty$, the instance must be infinite.

Refining \cref{que:local ext def}, we can define $\ext(T,\phi,\kappa)$ as the minimal $\lambda$ (if exists) such that whenever $M \vDash T$ and $X \subseteq M^k$ is externally definable by $\phi(x,c)$ for some $c\in \U$, then $X$ contains an $M$-definable subset of size $\geq \kappa$. By \cref{thm:main theorem paper}, if $T$ is NIP then $\ext(T,\phi,\kappa) \leq \kappa^+$. If the honest definition of $\phi$ is NIP then by \cref{rem:aleph_alpha+omega}, if $\kappa = \aleph_\alpha$, $\ext(T,\phi,\kappa)\leq \aleph_{\alpha+\omega}$. If we assume only that $\phi$ is NIP, it is not even clear that $\ext(T,\phi,\aleph_0)$ exists.

\begin{question}
  What is $\ext(T,\phi,\kappa)$ when $\phi$ is NIP?
\end{question}

\begin{remark}\label{rem:random}
  Let $T$ be a complete theory. Suppose that there is some infinite $\emptyset$-definable set $Z$ of $x$-tuples such that $\phi(x,y)$ is \defn{random} on $Z$: for any finite disjoint sets $A,B\subseteq Z$ there is some $y$-tuple $d$ such that $A = \phi(A\cup B,d)$ (for the notations, see \cref{subsec:notation}). This is a strong negation of NIP and happens e.g., in the case of the random graph. Then every subset of $Z$ is externally definable by compactness. Let $T^{\Sk}$ be a Skolemization of $T$. Let $\lambda$ be any infinite cardinal and let $I = \sequence{a_i}{i< \lambda}$ be an indiscernible sequence (in the sense of $T^{\Sk}$) contained in $Z$ (in some model of $T$), and let $N = \Sk(I)$ (the Skolem hull of $I$). Then $X := \set{a_i}{i \text{ even}}$ is a subset of $N$ which is externally definable by an instance of $\phi$, but which does not contain an infinite $N$-definable subset (even in $\L^{\Sk}$). Hence $\ext(T,\phi,\aleph_0) = \infty$, and in particular $\ext(T, \aleph_0) = \infty$.
\end{remark}

\begin{question}
  Does $\ext(T,\aleph_0) = \infty$ hold whenever $T$ is IP?
  That is, does every IP theory have a model containing an uncountable externally definable set which contains no infinite definable set?
\end{question}

\appendix
\section{Almost agreeing orders on the countable ordinals} \label{appendix}
In this appendix, we show how to construct on each countable ordinal an order of order type $\omega$, in such a way that any two of the orders agree up to a finite set.
This result is not used in the paper.
It formed part of our first attempt to prove \cref{thm:cofinal family of VC dim 2},
but in the end turned out not to provide a route to proving that theorem.
We nonetheless present the result in this appendix, in the hope that it may be of interest in its own right.

\begin{definition}
  Let $X$ be a set. Say two orders $<^1$ and $<^2$ on $X$ \defn{almost agree}, and write
  $\mathord{<^1} \sim \mathord{<^2}$, if there is a finite subset $X_0 \subseteq X$ such that $\mathord{<^1}|_{(X\setminus X_0)} = \mathord{<^2}|_{(X\setminus X_0)}$.
\end{definition}

Note that $\sim$ is an equivalence relation.

If $(X,<)$ has order type $\omega$, we call $<$ an $\omega$-order on $X$.

\begin{theorem}\label{thm:almost agreeing orders on aleph1}
  There are $\omega$-orders $<^\alpha$ on each $\alpha$ for $\omega \leq \alpha<\omega_1$ such that
  $\mathord{<^\beta} \sim \mathord{<^\alpha} |_{\beta}$ whenever $\omega \leq \beta<\alpha$.
\end{theorem}

Before proving \cref{thm:almost agreeing orders on aleph1}, we establish a pair of lemmas.

\begin{lemma} \label{lem:adjusing one order}
  Suppose that $(X,<^X)$ and $(Y,<^Y)$ are both $\omega$-orders, $X\subseteq Y$, and $\mathord{<^X} \sim \mathord{<^Y}|_X$. Then there is some $\omega$-order $\lessdot^Y$ on $Y$ such that $\mathord{\lessdot^Y} \sim \mathord{<^Y}$ and $\mathord{\lessdot^Y}|_X = \mathord{<^X}$.
\end{lemma}

\begin{proof}
  Let $X_0 \subseteq X$ be finite such that $<^X$ and $\mathord{<^Y} |_X$ agree on $X \setminus X_0$.
  We define an order on $Y$ which agrees with $<^Y$ on $Y \setminus X_0$ and places $X_0$ in a way which agrees with $<^X$ on $X$.
  Formally, we prove the lemma by induction on $|X_0|$. If $X_0$ is empty there is nothing to do. Let $x \in X_0$, $Z = X\setminus \{x\}$, $W = Y \setminus \{x\}$, $\mathord{<^Z} = \mathord{<^X} |_Z$ and $\mathord{<^W} = \mathord{<^Y} |_W$. Note that $<^Z$ and $<^W$ are still $\omega$-orders. By the induction hypothesis, there is some order $\lessdot^W$ on $W$ such that $\mathord{\lessdot^W} \sim \mathord{<^W}$ and $\mathord{<^Z} \subseteq \mathord{\lessdot^W}$.

  Let $F = \set{y\in W}{\exists z \in X\; (z <^X x \land y \leqdot^W z)}$; this is the cut on $(W,\lessdot^W)$ induced by the cut of $x$ on $(Z,<^Z)$. Note that $F$ is downwards closed in $\lessdot^W$ and that $F$ is finite: let $x' \in X$ be such that $x<^X x'$. Then if $y\in F$ and $z$ witnesses this, then $z <^X x <^X x'$ so that $z<^Z x'$ and hence $y \leqdot^W z \lessdot^W x'$. But $(W,\lessdot^W)$ is an $\omega$-order and hence $\set{y\in W}{y\lessdot^W x'}$ is finite.

  Let $\lessdot^Y$ extend $\lessdot^W$ and be such that for all $y\in Y$, $y\lessdot^Y x$ iff $y \in F$. To show that $\lessdot^Y$ is an $\omega$-order, it is enough to show that $\set{y\in Y}{y<^Y x}$ is finite, but this is precisely $F$. Since $F \cap X = \set{z\in Z}{z<^Z x}= \set{x'\in X}{x'<^X x}$, it follows that $\mathord{\lessdot^Y} \supseteq \mathord{<^X}$. Finally, if $\lessdot^W$ and $<^W$ agree on $W \setminus W_0$ where $W_0 \subseteq W$ is finite, then $<^Y$ and $\lessdot^Y$ agree on $Y\setminus (W_0 \cup \{x\})$ so that $\mathord{<^Y}\sim \mathord{\lessdot^Y}$ as required.
\end{proof}

\begin{lemma}\label{lem: adjusing an infinite sequence}
  Suppose that $\sequence{X_i}{i<\omega}$ is an increasing sequence of countable sets, $(X_i,<^i)$ are $\omega$-orders, and $\mathord{<^{i+1}} |_{X_i} \sim \mathord{<^i}$ for all $i<\omega$. Then there are $\omega$-orders $\lessdot^i$ on $X_i$ such that $\mathord{\lessdot^0} = \mathord{<^0}$, and $\mathord{\lessdot^i} \sim \mathord{<^i}$, and $\mathord{\lessdot^{i+1}} |_{X_i} = \mathord{\lessdot^i}$ for all $i<\omega$.
\end{lemma}

\begin{proof}
  Inductively define $\lessdot^i$ as follows. Let $\mathord{\lessdot^0} = \mathord{<^0}$. Suppose we defined $\lessdot^i$. Since $\mathord{<^i} \sim \mathord{\lessdot^i}$ and $\mathord{<^{i+1}} |_{X_i} \sim \mathord{<^i}$, it follows that $\mathord{<^{i+1}} |_{X_i} \sim \mathord{\lessdot^i}$. By \cref{lem:adjusing one order}, there is some $\omega$-order $\lessdot^{i+1}$ on $X_{i+1}$ such that $\mathord{\lessdot^{i+1}} \sim \mathord{<^{i+1}}$ and $\mathord{\lessdot^{i+1}} |_{X_i} = \mathord{\lessdot^i}$, as required.
\end{proof}

\begin{proof}[Proof of \cref{thm:almost agreeing orders on aleph1}]
  We define the orders $<^\alpha$ by induction on $\omega \leq \alpha <\omega_1$.
  Define $<^\omega$ as the canonical order on $\omega$.

  Suppose that $\alpha = \beta +1$ and $<^\beta$ has been defined. Let $\mathord{<^\alpha} = \mathord{<^\beta} \cup \set{(\beta,\gamma)}{\gamma<\beta}$. In other words, we put $\beta$ as the first element of $<^\alpha$ without changing anything else. Now, if $\gamma\leq \beta$ then $\mathord{<^\alpha} |_{\gamma} = \mathord{<^\beta} |_{\gamma} \sim \mathord{<^\gamma}$ by induction.

  Now suppose that $\alpha>\omega$ is a limit ordinal. Let $\sequence{\alpha_i}{i<\omega}$ be an increasing sequence of ordinals, cofinal in $\alpha$, where $\alpha_0 = \omega$. Apply \cref{lem: adjusing an infinite sequence} to the sequence $(\alpha_i, <^{\alpha_i})$ (which we can by the induction hypothesis) to get $\omega$-orders $\lessdot^{\alpha_i}$ on $\alpha_i$ such that $\mathord{\lessdot^{\alpha_0}} = \mathord{<^{\alpha_0}} = \mathord{<^\omega}$, and $\mathord{\lessdot^{\alpha_i}}\sim \mathord{<^{\alpha_i}}$, and $\mathord{\lessdot^{\alpha_{i+1}}} \supseteq \mathord{\lessdot^{\alpha_i}}$.

  %
  %
  %

  Let $\mathord{<^*} = \bigcup\set{\mathord{\lessdot^{\alpha_i}}}{i<\omega}$. We define an $\omega$-order $<^\alpha$ on $\alpha$ by, roughly speaking, concatenating the finite orders obtained by taking, for each $i<\omega$ in turn, those elements of $\alpha_i$ which are $\lessdot^{\alpha_i}$-less than $i$ and have not yet been taken, ordered by $\lessdot^{\alpha_i}$. Formally: for $\beta <\alpha$, let $(-\infty,\beta)$ be $\set{\gamma <\alpha}{\gamma <^* \beta}$. Define inductively sets $b_i \subseteq \alpha$ for $i<\omega$ as follows: $b_i = \alpha_i\cap (-\infty,i) \setminus \bigcup_{j< i}{b_j}$ (so $b_0 = \emptyset = (-\infty,0)\cap \alpha_0$, since $\mathord{\lessdot^{\alpha_0}} = \mathord{<^{\omega}}$ is the canonical order on $\omega$). Note that $b_i \cap b_j =\emptyset$ for $i \neq j$ and that $\alpha = \bigcup_{i<\omega}b_i$ (if $\beta<\alpha$, then $\beta <\alpha_i$ for some $i<\omega$ and hence $\beta \lessdot^{\alpha_i} m$ for some $m<\omega$, so that $\beta \in (-\infty,m)$ and hence $\beta \in b_k$ for some $k\leq \max\{i,m\}$). Finally, note that each $b_i$ is finite since $(\alpha_i,\lessdot^{\alpha_i})$ has order type $\omega$.

  Order each $b_i$ by $<^*$ and put the $b_i$'s in order to define $<^\alpha$. More formally, let $<^{lex}$ be the lexicographical order on $\omega \times \alpha$ (taking the canonical order on $\omega$ and $<^*$ on $\alpha$). For $\beta<\alpha$ let $i(\beta)$ be such that $\beta \in b_{i(\beta)}$, and for $\beta,\gamma <\alpha$ put $\gamma <^{\alpha} \beta$ iff $(i(\beta),\beta) <^{lex} (i(\gamma),\gamma)$.

  We check that $<^\alpha$ is as required. The order type of $(\alpha,<^*)$ is $\omega$ since each $b_i$ is finite, so that for any $\beta <\alpha$, $\set{\gamma<\alpha}{\gamma <^{\alpha} \beta}$ is finite. Now suppose that $\beta < \alpha$. Then $\beta <\alpha_i$ for some $i<\omega$. Since $\mathord{<^\beta}\sim \mathord{<^{\alpha_i}} |_{\beta}$ and $\mathord{<^{\alpha_i}} \sim \mathord{\lessdot^{\alpha_i}}$, to show $\mathord{<^\beta} \sim \mathord{<^\alpha} |_{\beta}$ it suffices to check that $\mathord{\lessdot^{\alpha_i}} \sim \mathord{<^\alpha}|_{\alpha_i}$.

  To show this we show that if $\gamma,\beta \in \alpha_i \setminus \bigcup_{j\leq i}b_j$ then
  \begin{equation}\label{eqn:showing that < is good}
    \gamma \lessdot^{\alpha_i} \beta \iff\gamma <^\alpha \beta. \tag{*}
  \end{equation}
  Indeed, if $\gamma,\beta \in b_j$ for some $j<\omega$ then \cref{eqn:showing that < is good} follows from the fact that $<^\alpha$ equals $<^*$ on $b_j$, so extends $\lessdot^{\alpha_i} |_{b_j}$. Suppose that $\gamma \in b_j$, $\beta \in b_k$ and $j \neq k$, so without loss $i<j<k$. Then, since $\gamma,\beta \in \alpha_i$, necessarily $\gamma \in (-\infty,j)$ and $\beta \notin (-\infty,j)$. In this case, $\gamma \lessdot^{\alpha_i} j \lessdot^{\alpha_i} \beta$ (since this is true for $<^*$) and $\gamma <^\alpha \beta$ by definition.
\end{proof}

\bibliographystyle{alpha}
\bibliography{externallydefinable}

\end{document}